\pdfoutput=1

\documentclass[10pt,reqno, twoside]{amsart}
\usepackage[letterpaper,margin=1.2in,footskip=0.70in]{geometry}

\usepackage{lmodern}
\usepackage[protrusion=true,expansion=true]{microtype}

\usepackage[inline,shortlabels]{enumitem}
\setlist{noitemsep}

\usepackage{amsmath,amssymb,amsthm}
\usepackage{mathtools}
\usepackage{colonequals,extarrows}
\usepackage{tikz-cd}
\usepackage{array}

\usepackage[pdfusetitle,hidelinks,pagebackref]{hyperref}
\usepackage[capitalise,noabbrev]{cleveref}
\crefformat{equation}{\ensuremath{(#2#1#3)}}
\crefmultiformat{equation}{\ensuremath{(#2#1#3)}}{ and~\ensuremath{(#2#1#3)}}{, \ensuremath{(#2#1#3)}}{, and~\ensuremath{(#2#1#3)}}

\numberwithin{figure}{section}
\numberwithin{equation}{section}

\newtheorem{maintheorem}{Theorem}

\newtheorem{lemma}[figure]{Lemma}
\newtheorem{corollary}[figure]{Corollary}
\newtheorem{proposition}[figure]{Proposition}
\theoremstyle{definition}
\newtheorem{definition}[figure]{Definition}
\newtheorem{notation}[figure]{Notation}
\theoremstyle{definition}
\newtheorem{remark}[figure]{Remark}
\theoremstyle{definition}
\newtheorem{example}[figure]{Example}
\theoremstyle{definition}
\newtheorem{construction}[figure]{Construction}
\theoremstyle{definition}
\newtheorem{question}[figure]{Question}

\mathchardef\mhyphen="2D

\newcommand{\et}{\mathrm{\acute{e}t}}
\newcommand{\Fib}{\mathrm{Fib}}
\newcommand{\fpqc}{\mathrm{fpqc}}
\DeclareMathOperator{\Fun}{Fun}
\DeclareMathOperator{\Grp}{Grp}
\DeclareMathOperator{\Hom}{Hom}
\DeclareMathOperator{\im}{im}
\DeclareMathOperator{\Ran}{Ran}
\DeclareMathOperator{\LMod}{LMod}

\newcommand{\mult}{\times}
\DeclareMathOperator{\Ob}{Ob}
\newcommand{\op}{\mathrm{op}}
\DeclareMathOperator{\Post}{Post}
\DeclareMathOperator{\pr}{pr}
\newcommand{\pre}{\mathrm{pr}}
\newcommand{\proet}{{\mathrm{pro\acute et}}}
\DeclareMathOperator{\PShv}{PShv}

\newcommand{\red}{\mathrm{red}}
\DeclareMathOperator{\Shv}{Shv}
\DeclareMathOperator{\Sing}{Sing}
\DeclareMathOperator{\Spec}{Spec}
\DeclareMathOperator{\tot}{tot}

\newcommand{\EE}{\mathbf{E}}
\newcommand{\NN}{\mathbf{N}}
\newcommand{\RR}{\mathbf{R}}

\newcommand{\ZZ}{\mathbf{Z}}

\newcommand{\cC}{\mathcal{C}}

\newcommand{\cL}{\mathcal{L}}
\newcommand{\cO}{\mathcal{O}}
\newcommand{\cT}{\mathcal{T}}
\newcommand{\cX}{\mathcal{X}}

\newcommand{\Ani}{\mathcal{S}}

\newcommand{\Set}{\mathrm{Set}}
\newcommand{\Sp}{\mathrm{Sp}}

\newcommand\restr[2]{{\left.\kern-\nulldelimiterspace#1\vphantom{\big|}\right|_{#2}}}
\newcommand{\suchthat}{\;\ifnum\currentgrouptype=16 \middle\fi\vert\;}

\newcolumntype{C}[1]{>{\centering\arraybackslash}p{#1}}

\AtBeginDocument{
    \def\MR#1{}
}

\begin{document}

\title{On Postnikov completeness for replete topoi}

\subjclass[2020]{18N60, 18F10, 14F06}

\author{Shubhodip Mondal}
\author{Emanuel Reinecke}

\address[Shubhodip Mondal]{Department of Mathematics, University of British Columbia, 1984 Mathematics Road, Vancouver BC  V6T 1Z2, Canada}
\email{smondal@math.ubc.ca}

\address[Emanuel Reinecke]{School of Mathematics, Institute for Advanced Study, 1 Einstein Drive, Princeton, NJ 08540, USA}
\email{reinec@ias.edu}

\begin{abstract}
We show that the hypercomplete $\infty$-topos associated with any replete topos is Postnikov complete, positively answering a question of Bhatt and Scholze;
this will be deduced from the Milnor sequences for sheaves of spaces on replete topoi that we construct. As a corollary, we generalize a result of To\"en on affine stacks.
\end{abstract}

\maketitle

\section{Introduction}
In their seminal work \cite{proet}, Bhatt--Scholze introduced the pro-\'etale site of a scheme, which simplifies many foundational constructions in the theory of $\ell$-adic cohomology and has been fundamental to the recent development of condensed mathematics \cite{BH,cond}.
One of the key properties of the associated $1$-topos is captured by the notion of replete topoi, which they study in \cite[\S~3]{proet}.
We recall this definition here.
\begin{definition}[{\cite[Def.~3.1.1]{proet}}]\label{introdef} A $1$-topos $\cX$ is \emph{replete} if for every diagram $F \colon \ZZ_{\ge 0}^{\op} \to \cX $ with the property that $F_{n+1} \to F_n$ is surjective for all $n$, the natural map $\lim F \to F_n$ is surjective for every $n$.
\end{definition}
Sheaves of spaces on a $1$-topos $\cX$ (see e.g.\ \cite[\S~6.5]{HTT}) form an $\infty$-topos, which we denote by $\Shv_{\infty}(\cX)$.
Let us denote the associated hypercomplete $\infty$-topos by $\Shv_{\infty}(\cX)^{\wedge}$.
One may think of objects of $\Shv_{\infty}(\cX)^{\wedge}$ as sheaves of spaces that satisfy hyperdescent.
In \cite[Qn.~3.1.12]{proet}, the authors ask the following question:
\begin{question}\label{quesintro}
Let $\cX$ be a replete topos.
Do Postnikov towers converge in the hypercomplete $\infty$-topos $\Shv_{\infty}(\cX)^{\wedge}$? 
\end{question}

One important feature of the above question is that it does not impose any finiteness assumptions (such as homotopy dimension $\le n$ \cite[Def.~7.2.1.1]{HTT} or cohomological dimension $\le n$ \cite[Cor.~7.2.2.30]{HTT}) on the $\infty$-topoi that appear in the previously known criteria for convergence of Postnikov towers (see \cref{finiteness} and \cref{cry}).

Certain stable analogs of \cref{quesintro} can already be found in the literature.
In \cite[Prop.~3.3.3]{proet}, the authors show that if $\cX$ is a replete topos, the derived $\infty$-category $D(\cX,\ZZ)$ (or in other words, hypercomplete sheaves of $H\ZZ$-module spectra) is left complete.
In \cite[Prop.~A.10]{akhil}, the case of sheaves of spectra on certain large sites (which satisfy an additional ``quasi-compactness'' condition and admit countable filtered limits) is addressed and used in the context of the arc-topology.
Regarding \cref{quesintro} itself, in \cite[Prop.~3.2.3]{proet}, the authors show that Postnikov towers converge in $\Shv_{\infty} (\cX)^{\wedge}$, when $\cX$ is additionally assumed to be locally weakly contractible (see \cite[Def.~3.2.1]{proet}).
The pro-\'etale topos of a scheme is an example of a locally weakly contractible topos \cite[Prop.~4.2.8]{proet}.
In this paper, we answer \cref{quesintro} in general.
\begin{maintheorem}\label{mainthm}
 Let $\cX$ be a replete topos.
 Then the hypercomplete $\infty$-topos $\Shv_{\infty}(\cX)^{\wedge}$ is Postnikov complete.
\end{maintheorem}
\begin{remark}\label{rmkmore}
In \cref{mainthm}, we prove that $\Shv_{\infty}(\cX)^{\wedge}$ is Postnikov complete in the sense of \cref{defconv};
this is stronger than just requiring that the natural map $X \to \lim_{n} \tau_{\le n}X$ is an equivalence for every $X \in \cX$. 
\end{remark}

Let us briefly explain the main ingredients in the proof of \cref{mainthm}.
Let $\cX$ be a replete topos. We need to prove that any tower in $\Shv_{\infty} (\cX)^\wedge$ which ``looks like" a Postnikov tower is in fact a Postnikov tower (see \cref{equivalentpost}). 
As explained in detail in the proof of \cref{mainthm}, this follows more or less directly from Milnor sequences, which allow one to calculate homotopy groups of an $\NN$-indexed inverse limit in the expected way.

However, it is quite subtle to have a theory of Milnor sequences in general since sheafification does not commute with infinite limits (see \cite[p.~21]{Toe} and \cref{toeremark}). In fact, it is already unclear if $\pi_0$ commutes with $\NN$-indexed direct products in $\Shv_{\infty}(\cX)$. 
Let us consider the example arising from condensed mathematics when $\cX$ is the pro-\'etale site of a point \cite[Def.~1.2]{cond}. In that case, one can get around this issue by realizing that a sheaf is essentially the same as a presheaf on extremally disconnected sets that sends finite disjoint unions to products \cite[Prop.~2.7]{cond}. 
Roughly speaking, presheaves on the pro-\'etale site of a point that send coproducts to products are essentially as good as sheaves.

Motivated by this property, for a general replete topos we work with a class of abstract presheaves called \emph{multiplicative presheaves} (see \cref{subca}), which send arbitrary disjoint unions to products. We observe in \cref{product-sheafification} that sheafification commutes with $\NN$-indexed products of multiplicative presheaves, overcoming the subtleties discussed above and leading to the existence of Milnor sequences (see \cref{milreplete}).

\begin{remark}
Let $\cX$ be a replete topos.
One might wonder whether $\Shv_{\infty}(\cX)$ is automatically hypercomplete.
In \cref{counterexample}, we show that this is false even if one assumes $\cX$ to be locally weakly contractible.
This also implies that $\Shv_{\infty}(\cX)$ is not Postnikov complete in general (see \cref{darm1}).
\end{remark}
\begin{example}\label{replete-not-finite}
Many topoi naturally appearing in algebraic geometry are replete.
\Cref{mainthm} is applicable in all such contexts, without having to make any assumptions on the finiteness of homotopy or cohomological dimension.
For instance (after fixing set-theoretic issues, see e.g.\ the footnote in \cite[Ex.~3.1.7]{proet}), topoi coming from the fpqc-topology, $v$-topology (see e.g.\ \cite[\href{https://stacks.math.columbia.edu/tag/0EVM}{Tag~0EVM}]{stacks-project}), and the quasisyntomic topology \cite[Def.~ 4.10]{BMS2} are all replete, but in general not of finite cohomological or homotopy dimension:
as a concrete example, arc-descent for the \'etale cohomology of torsion sheaves \cite[Thm.~5.4]{arc} shows that
\[ H^*(\Spec(\RR)_\fpqc,\ZZ/2) \simeq H^*(\Spec(\RR)_v,\ZZ/2)\simeq H^*(\Spec(\RR)_\et,\ZZ/2). \]
The latter is computed by the group cohomology ring $H^* (\ZZ/2, \ZZ/2)$, which is a symmetric algebra on $H^1(\ZZ/2, \ZZ/2)$ and thus nontrivial in all degrees (\textit{cf}.~\cref{exam}). On the other hand, smaller topoi whose covers are subject to certain finiteness conditions (such as Zariski, \'etale, fppf) tend to be not replete (\textit{cf}.~\cite[Ex.~3.1.5]{proet}).
\end{example}
\begin{example}[{\cite[\S~4.3]{proet}}]
Let $G$ be a profinite group.
Let $(BG)_{\proet}$ denote the site of profinite sets with a continuous $G$-action, with covers given by continuous surjections. Then $\Shv\bigl((BG)_{\proet}\bigr)$ is replete. As a special case of \cref{mainthm}, $\Shv_{\infty}\bigl((BG)_{\proet}\bigr)^{\wedge}$ is Postnikov complete.
However, this fails for non-replete variants of this topos that do not take the profinite topology into account;
\textit{cf.}~\cref{exam}.
\end{example}
To exhibit another consequence of \cref{mainthm}, let us mention the following application to affine stacks in the sense of \cite[\S~2.2]{Toe}.
\begin{corollary}[\cref{lastcor}]\label{maincor}
Let $F$ be an affine stack over $\Spec B$ for any ring $B$.
Then the natural map $F \to \lim_{n} \tau_{\le n} F$ is an equivalence.
\end{corollary}
Previously, the above corollary was only known under certain assumptions (\cref{lastrmk}) and the proof relied on the vanishing of cohomology groups of affine schemes with coefficients in unipotent group schemes in degrees $>1$.
As we see now, a more general result follows from \cref{mainthm} simply as a consequence of repleteness.
The Milnor sequences for replete topoi that appears in this paper will also be used in \cite{soon}.
\subsection*{Acknowledgments}
We thank Bhargav Bhatt, Peter Haine, Akhil Mathew, and Peter Scholze for helpful comments and conversations. Special thanks are due to the referee for many suggestions that led to simplifications of the proofs and the exposition.
We are grateful to the Max Planck Institute for Mathematics (Bonn, Germany) and the Institute for Advanced Studies (Princeton, USA) for their support during the preparation of this work. 
Additionally, Mondal acknowledges support from the University of Michigan, the NSF Grant DMS \#1801689, FRG \#1952399 and the University of British Columbia (Vancouver, Canada) and Reinecke acknowledges support from the NSF Grant DMS \#1926686.

\addtocontents{toc}{\protect\setcounter{tocdepth}{2}}

\section{Multiplicative presheaves on Grothendieck sites}

In this section, we introduce a key class of objects for our paper, which we call ``multiplicative presheaves.''
Roughly speaking, they capture the notion of a presheaf that takes arbitrary coproducts to products.
However, a site may fail to have arbitrary coproducts.
To make the notion precise, we will therefore need to phrase this condition in the appropriate category.
Before formally introducing multiplicative presheaves, let us fix some notation.
\begin{notation}\label{leftkan}
Let $\cT$ be a Grothendieck site. The Yoneda embedding gives a natural fully faithful functor $h \colon \cT \to \PShv(\cT)$. Composition with the sheafification functor yields a functor $h^{\sharp} \colon \cT \to \Shv(\cT)$.
If there is no risk of confusion, we write $h_X \colonequals h(X)$ and $h^\sharp_X \colonequals h^\sharp(X)$ for $X \in \cT$. We may view a presheaf $P$ on $\cT$ as a functor $P \colon \cT^\op \to \Set$.
By right Kan extension along $h^\sharp$, we obtain an extended functor $\Ran_{h^\sharp}(P) \colon \Shv(\cT)^\op \to \Set$.
\end{notation}
\begin{lemma}
Let $\cT$ be a Grothendieck site and $X \in \cT$.
Then for all $F \in \Shv(\cT)$, we have $\bigl(\Ran_{h^\sharp}(h_X)\bigr)(F) = \Hom_{\Shv(\cT)}\bigl(F,h^\sharp_X\bigr)$;
that is, $\Ran_{h^\sharp}\bigl(h_X\bigr)$ is represented by $h^\sharp_X$.
\end{lemma}
\begin{proof}
Follows from the universal property of right Kan extensions.
For example, see \cite[Exercise~X.3.2]{Cats}, which states that the right Kan extension of a (co)representable functor must be (co)represented by the image of the former (co)representing object.
\end{proof}
The following two lemmas are well known and follow from the universal property of right Kan extensions and the adjoint functor theorem.
\begin{lemma}\label{lemma23}
Let $\cT$ be a Grothendieck site. If a functor $P \colon \Shv(\cT)^\op \to \Set $ preserves all small limits, then $P \circ h^\sharp \colon \cT^\op \to \Set$ is a sheaf. Further, there is a natural isomorphism $\Ran_{h^\sharp}(P \circ h^\sharp) \simeq {P}$.
\end{lemma}
\begin{lemma}\label{lemma2}
Let $\cT$ be a Grothendieck site. If $P \colon \cT^\op \to \Set$ is a sheaf, then $\Ran_{h^\sharp}(P) \colon \Shv(\cT)^\op \to \Set$ preserves all small limits.
Further, there is a natural isomorphism $\Ran_{h^\sharp}(P)\circ h^\sharp \simeq P$.
\end{lemma}
Now we are ready to introduce the notion of multiplicative presheaves.
\begin{definition}[Multiplicative presheaves]\label{subca}
Let $\cT$ be a Grothendieck site.
Let $P$ be a presheaf on $\cT$. We say that $P$ is a \emph{multiplicative presheaf} if for every set of objects $\left \{X_j\right \}_{j \in J} \in \cT$ for an indexing set $J$, the natural map of presheaves $\coprod_{j \in J} h_{X_j} \to \left(\coprod_{j \in J} h_{X_j}\right)^\sharp$ induces an isomorphism \[ \Hom_{\PShv(\cT)} \biggl(  \biggl( \coprod_{j \in J} h_{X_j} \biggr)^\sharp, P \biggr) \xlongrightarrow{\sim} \Hom_{\PShv(\cT)} \biggl (   \coprod_{j \in J} h_{X_j}, P \biggr) \simeq \prod_{j \in J} P (X_j). \]
\end{definition}
\begin{remark}\label{evening}
It follows directly from the definition that limits of multiplicative presheaves are again multiplicative.
\end{remark}
\begin{example}\label{ia}
By the universal property of sheafification, every sheaf is an example of a multiplicative presheaf.
\end{example}
The following proposition provides a useful source of multiplicative presheaves.
\begin{proposition}\label{local-objects}
Let $\cT$ be a Grothendieck site and $P' \colon \Shv(\cT)^\op \to \Set$ is a product preserving functor. Let $P \colon \cT^\op \to \Set$ denote the functor $P'\circ h^\sharp$. 
Then $P$ is a multiplicative presheaf.
\end{proposition}
\begin{proof}
By the universal property of right Kan extensions, we have a natural bijection
\[ \Hom\biggl(\biggl(\coprod_{j \in J} h_{X_j}\biggl)^\sharp, P\biggr) \simeq  \Hom\biggl(\Ran_{h^\sharp}\biggl(\biggl(\coprod_{j \in J} h_{X_j}\biggr)^\sharp\biggr), P'\biggr). \]
However, \cref{lemma23} implies that $\Ran_{h^\sharp}\bigl(\bigl(\coprod_{j \in J} h_{X_j}\bigr)^\sharp\bigr)$ is represented by $\bigl(\coprod_{j \in J} h_{X_j}\bigr)^\sharp$.
This shows that the right-hand side in the above isomorphism is further in natural bijection with
\[  P'\biggl(\biggl(\coprod_{j \in J} h_{X_j}\biggr)^\sharp\biggr) \simeq \prod_{j \in J}\bigl(P' \circ h^{\sharp}\bigr) (X_j)   \simeq  \prod_{j \in J} P (X_j), \]
where the first isomorphism follows from the facts that $P'$ preserves products and sheafification preserves all colimits, and the second one follows by definition of $P$.
Therefore, the natural map $\Hom\bigl(\bigl(\coprod_{j \in J} h_{X_j}\bigr)^\sharp, P \bigr) \to \prod_{j \in J} P (X_j)$ is an isomorphism, which shows that $P$ is multiplicative, as desired.
\end{proof}
\begin{example}
Let $\cT$ be a Grothendieck site and let $F$ be a sheaf of abelian groups on $\cT$. For any integer $n \ge 0$, the functor that sends an object $c \in \Ob(\cT)$ to the group $H^n (c,F)$ is a multiplicative presheaf on $\cT$. This follows from \cref{cocacola}, which will be proven later.
By the exactness of the sheafification functor, the sheafification of the multiplicative presheaf $c \mapsto H^n(c,F)$ is zero for $n \ge 1$.
\end{example}
\begin{example}
Let $\cT$ be the Grothendieck site of schemes with the $v$-topology. Let $\cO$ denote the structure presheaf  on $\cT$. We claim that $\cO$ is not multiplicative. Indeed, for any scheme $X$, the closed immersion of the reduced subscheme $X_\red \hookrightarrow X$ is a cover in the $v$-topology.
Since the associated morphism of representable presheaves becomes an effective epimorphism after sheafification,
\[ h^\sharp_{X_\red \times_X X_\red} \rightrightarrows h^\sharp_{X_\red} \to h^\sharp_X \]
is a coequalizer diagram of $v$-sheaves.
However, as $X_\red \hookrightarrow X$ is a closed immersion, $X_\red \times_X X_\red \simeq X_\red$ and the two projections $X_\red \times_X X_\red \to X_\red$ are identified with the identity.
Thus, $h^\sharp_{X_\red} \xrightarrow{\sim} h^\sharp_X$ is an isomorphism for all $X$.
In particular, for a multiplicative presheaf $P$ the natural map $P (X) \to P(X_{\red})$ must be an isomorphism.
On the other hand, when $X = \Spec A$ for a nonreduced ring $A$, we have $\cO(X) = A \neq A_\red = \cO(X_\red)$, so that $\cO$ cannot be a multiplicative presheaf. 
\end{example}
\begin{remark}
Let us denote the category of multiplicative presheaves by $\PShv_{\mult}(\cT)$. 
There is a natural inclusion functor $i_1 \colon \PShv_{\mult}(\cT) \to \PShv(\cT)$.
As noted in \cref{evening}, the category $\PShv_{\mult}(\cT)$ has all small limits, and they are preserved by the inclusion functor $i_1$.
Similarly, the inclusion functor $i_2 \colon \Shv(\cT) \to \PShv_{\mult}(\cT)$ also preserves all small limits. Therefore, by the adjoint functor theorem, the sheafification functor $\PShv(\cT) \to \Shv(\cT)$
can be expressed as a composition of the left adjoints $i_1^{\cL} \colon \PShv(\cT) \to \PShv_{\mult}(\cT)$ and $i_2^{\cL} \colon \PShv_{\mult}(\cT) \to \Shv(\cT)$. Our main observation, which is recorded in \cref{lim-sheafification}, is that when $\Shv(\cT)$ is replete, the functor $i_2^{\cL}$ enjoys good formal properties.
\end{remark}
We now proceed towards proving the main result regarding multiplicative presheaves (\cref{lim-sheafification}): we will show that sheafification commutes with $\NN$-indexed inverse limits of multiplicative presheaves. To this end, we will need a technical lemma.
\begin{lemma}\label{lemmasurj}
Let $\cT$ be a Grothendieck site such that the associated $1$-topos $\Shv(\cT)$ is replete. Let $h \colon (F_i)_{i \in \NN} \to (G_i)_{i \in \NN}$ be a map of $\NN$-indexed inverse systems of multiplicative presheaves on $\cT$. Suppose that for each $i \in \NN$, the induced map $u_i \colon F^\sharp_i \to G^\sharp_i$ on sheafifications is an isomorphism. Then the natural map
\[ \theta \colon (\lim_{i} F_i)^\sharp \to (\lim_{i} G_i)^\sharp \]
is a surjection of sheaves.
\end{lemma}{}
\begin{proof}
We will denote by $h_i \colon F_i \to G_i$ the maps of presheaves occurring in the map of inverse systems $h$.
Let $\pi_i \colon F_i \to F_{i-1}$ and $\gamma_i \colon G_i \to G_{i-1}$ denote the transition maps of the inverse systems.
Let us pick $X \in \mathcal{T}$ and an element $(\dotsc, s_1, s_0)=(s_i) \in \lim_i G_i(X)$.
To prove the surjectivity of $\theta$, it suffices to show that the image of $(s_i)$ in $(\lim G_i)^\sharp (X)$ can be lifted to $(\lim F_i)^\sharp$, after possibly refining $X$.

Since the map $u_0$ is an isomorphism, by definition of sheafification, one can choose a refinement $\{ X_{j_0} \to X \}_{j_0 \in I^0}$ and $t_{j_0} \in F_0 (X_{j_0})$ such that $h_0 (t_{j_0})= \restr{s_0}{X_{j_0}}$.
Next, using the hypothesis that $u_1$, $u_0$ are isomorphisms and that the $h_i$'s fit into a map of inverse systems, we can find a covering $\{ X_{j_1j_0} \to X_{j_0} \}_{j_1 \in I_{j_0}}$ and $t_{j_1j_0} \in F_1(X_{j_1j_0})$ such that $h_1 (t_{j_1j_0})=  \restr{s_1}{X_{j_1j_0}}$, and further satisfying the compatibility with the transition maps of the inverse system, i.e., $\pi_1(t_{j_1j_0}) = \restr{t_{j_0}}{X_{j_1j_0}}$.

Now we repeat the above procedure recursively. For every $n \ge 0$, we have a covering
\[ \left \{X_{j_{n+1} \dotso j_0} \to X_{j_n \dotso j_0} \right\}_{j_{n+1} \in I_{j_n \dotso j_0}} \] 
and sections $t_{j_{n+1}\dotso j_0} \in F_{n+1}(X_{j_{n+1}\dotso j_0})$ such that
\[ h_{n+1}(t_{j_{n+1}\dotso j_0}) = \restr{s_{n+1}}{X_{j_{n+1}\dotso j_0}} \quad \text{and} \quad \pi_{n+1}(t_{j_{n+1}\dotso j_0}) = \restr{t_{j_{n}\dotso j_0}}{X_{j_{n+1}\dotso j_0}}. \] 
Let us introduce a notation to keep track of the index sets systematically. 
Note that by virtue of our construction, the index set $I_{j_n\dotso j_0}$ depends on the choice of $j_0 \in I^0$, $j_1 \in I_{j_0}$,\ldots,$j_n \in I_{j_{n-1}\dotso j_0}$.
Let us denote all such possible choices of tuples $(j_n,\dotsc,j_0)$ by the set $I^n$. 
Then there are natural maps $w_n \colon I^{n+1} \to I^{n}$ such that fiber over $(j_n,\dotsc, j_0 )$ identifies with $I_{j_n\dotso j_0}$.
The sections $t_{j_{n+1} \dotso j_0} \in F_{n+1}(X_{j_{n+1} \dotso j_0})$ from before can simply be classified by a map
\[ \coprod_{\alpha \in I^{n+1}} h_{X_{\alpha}} \to F_{n+1}. \]

Since $F_{n+1}$ is a multiplicative presheaf, the above map factors uniquely through 
\[ \left( \coprod_{\alpha \in I^{n+1}} h_{X_{\alpha}}\right)^\sharp  \to F_{n+1}. \]
Moreover, by construction and the fact that the presheaves $F_i$ and $G_i$ are multiplicative, we obtain commutative diagrams
\[ \begin{tikzcd}
\left( \coprod_{\alpha \in I^{n+1}} h_{X_{\alpha}}\right)^\sharp  \arrow[rr] \arrow[d, "v_n"] &  & F_{n+1} \arrow[d, "\pi_{n+1}"] \arrow[rr, "h_{n+1}"] &  & G_{n+1} \arrow[d, "\gamma_{n+1}"] \\
\left( \coprod_{\alpha \in I^{n}} h_{X_{\alpha}}\right)^\sharp  \arrow[rr]                    &  & F_n \arrow[rr, "h_n"]                                &  & G_n.
\end{tikzcd} \]
Note that the maps $v_n$ are induced by the maps $w_n \colon I^{n+1} \to I^n$.
Thanks to \cite[\href{https://stacks.math.columbia.edu/tag/00WT}{Tag~00WT}]{stacks-project} and the stability of epimorphisms under coproducts, it follows that $v_n$ is an effective epimorphism. 
The above diagrams define maps of $\NN$-inverse systems on the category of presheaves on $\cT$. Let us denote the inverse limit of the left column (along the maps $v_n$) by $V$, which is automatically a sheaf. By virtue of our construction, we obtain a commutative diagram of sheaves
\[ \begin{tikzcd}
V \arrow[rr] \arrow[d] &  & (\lim F_i)^\sharp \arrow[d, "\theta"] \\
h_X^\sharp \arrow[rr]    &  & (\lim G_i)^\sharp.
\end{tikzcd} \]
Here, the map $h_X^\sharp \to (\lim G_i)^\sharp$ is obtained by sheafifiying the map $h_X \to (\lim G_i)$ classifying the chosen $(s_i) \in \lim G_i(X)$. 
By the repleteness of $\Shv(\cT)$, the map $V \to h_X^\sharp$ is an effective epimorphism. 
Therefore, after possibly refining $X$, one can lift the tautological section of $h_X^\sharp (X)$ to $V$ along the map $V \to h_X^\sharp$. Thus, after possibly refining $X$, the image of $(s_i)$ in $(\lim G_i)^\sharp (X)$ can be lifted to $(\lim F_i)^\sharp$. This finishes the proof.
\end{proof}{}
\begin{proposition}\label{lim-sheafification}
Let $\cT$ be a Grothendieck site such that the associated $1$-topos $\Shv(\cT)$ is replete.
Let $(P_n)_{n \in \NN}$ be an inverse system of multiplicative presheaves of sets on $\cT$.
Then the natural map $(\lim_n P_n)^\sharp \to \lim_n P^\sharp_n$ is an isomorphism.
\end{proposition}
\begin{proof}
The surjectivity of the natural map $\displaystyle{\theta \colon (\lim_n P_n)^\sharp \to \lim_n P^\sharp_n}$ follows directly from \cref{lemmasurj}. To prove that $\theta$ is injective, it suffices to show that the diagonal map is surjective, i.e., we wish to show that 
\[ (\lim P_n)^\sharp \to (\lim P_n)^\sharp \times_{\lim P_n^\sharp}(\lim P_n)^\sharp \]
is surjective.
Since sheafification commutes with finite limits, this is equivalent to showing that
\[ (\lim P_n)^\sharp \to \left (\lim P_n \times_{P_n^\sharp} P_n \right)^\sharp \]
is surjective.
By hypothesis, $P_n$ is a multiplicative presheaf. By \cref{evening} and \cref{ia}, the presheaf $P_n \times_{P_n^\sharp} P_n$ is also multiplicative. Since the natural map $P_n \to P_n \times_{P_n^\sharp} P_n$ induces an isomorphism after applying sheafification, we obtain the desired surjectivity by using \cref{lemmasurj}.
\end{proof}{}
As a consequence, we obtain similar results for sheafifications of countable direct products.
\begin{proposition}\label{product-sheafification}
Let $\cT$ be a Grothendieck site such that the associated $1$-topos $\Shv(\cT)$ is replete.
Then sheafification commutes with countable products of multiplicative presheaves on $\cT$.
\end{proposition}
\begin{proof}
The assertion follows formally since sheafification preserves finite limits and by \cref{lim-sheafification}, sheafification preserves $\NN$-indexed inverse limits of multplicative presheaves.
\end{proof}
\begin{construction}\label{constlim1}
We will construct a version of the $\lim_n^{1}$ functor in the generality of group objects in a $1$-topos $\cX$.
More precisely, we construct a functor
\[ {\lim_n}^{1} \colon \Fun(\NN^{\op}, \Grp(\cX)) \to \cX_*, \]
where the target denotes the category of pointed objects on $\cX$.
To this end, let $(P_n)_{n \in \NN}$ be an inverse system of group objects in a $1$-topos $\cX$.
Let $f_{n+1} \colon P_{n+1} \to P_n$ denote the transition maps.
Let $G$ denote the group $\prod_{n} P_n$ and $Z$ denote the object of $\cX$ underlying $G$.
Let $\alpha_n \colon G \to P_{n}$ and $\beta_n \colon Z \to P_{n}$ denote the projection maps.
Let us denote the composite map
\[ G \times Z \xrightarrow[]{\pr_2} Z \xrightarrow[]{\alpha_{n+1}} P_{n+1} \xrightarrow{f_{n+1}} P_{n} \to P_{n} \]
by $\psi_n$.
Here, the last map $ P_{n} \to P_{n}$ is the inverse operation on the group object $P_{n}$.
Let us consider the maps
\[ G \times Z \xrightarrow{\alpha_n \times \beta_n \times \psi_n} P_{n} \times P_{n} \times P_{n} \to P_{n}, \]
where the last map is induced by the (associative) multiplication operation on $P_n$.
Let us denote the composition of the above maps by $\theta_n \colon G \times Z \to P_n$. Using the maps $\theta_n$ for all $n$, we obtain a morphism
\[ \theta \colon G \times Z \to Z, \]
which defines an action of the group object $G$ on $Z$. This gives a diagram in $\cX$ of the form 
\begin{equation}\label{orbits} \begin{tikzcd}
  G \times Z \ar[r,shift left=.75ex,"\theta"]
  \ar[r,shift right=.75ex,swap,"\pr_2"] &Z.
\end{tikzcd} \end{equation}
We define $\lim_{n}^1 P_n$ to be the coequalizer of the above diagram. Note that $Z$ is naturally a pointed object of $\cX$ (via the unit morphisms $* \to P_n$ for all $n$); therefore $\lim_{n}^1 P_n$ is also naturally pointed and is an object of $\cX_*$.
\end{construction} 
\begin{remark}
\Cref{constlim1} above extends the classical construction for inverse systems of groups (see \cite[\S~IX.2]{yellow}) to the generality of group objects in topoi.
When $\cX$ is replete and the $P_n$ are abelian, it recovers the usual construction of $\lim^1_n$ as the first derived functor of $\lim_n$.
Indeed, in this case, $\lim^1_n P_n$ is the cokernel of the map
\[ (\alpha_n - f_{n+1} \circ \alpha_{n+1})_n \colon \prod_n P_n \to \prod_n P_n \]
(\textit{cf.}\ the proof of \cite[Prop.~3.1.11]{proet}), which is the coequalizer of (\ref{orbits}) in the category of abelian groups.
But $\theta$ and $\pr_2$ admit the common section $Z \simeq \{ 0 \} \times Z \hookrightarrow G \times Z$, so the coequalizer is reflexive.
Since sheaves in abelian groups are algebras over the ``free abelian sheaf'' monad on sheaves of sets and this monads preserves reflexive coequalizers, this is also the coequalizer in the category of sets;
\textit{cf.}\ \cite[Prop.~3]{Lin}.
Note that this argument does not apply to arbitrary sheaves of groups $P_n$ because in general $\theta$ is not a map of groups, so it does not make sense to talk about the coequalizer of (\ref{orbits}) in the category of groups.
\end{remark}

\begin{remark}
We point out that even though the inclusion functor from the category of sheaves to the category of presheaves on a site preserves limits, it does not preserve $\lim_{n}^1$ in general.
\end{remark}
\begin{proposition}\label{lim-sheafification2}
Let $\cT$ be a Grothendieck site such that the associated $1$-topos $\Shv(\cT)$ is replete.
Let $(P_n)_{n \in \NN}$ be an inverse system of multiplicative presheaves of groups on $\cT$.
Then the natural map $(\lim_n^{1} P_n)^\sharp \to \lim_n^{1} P^\sharp_n$ is an isomorphism.
\end{proposition}

\begin{proof}
The assertion follows formally from \cref{constlim1}, the fact that sheafification preserves small colimits (in particular, all coequalizers) and the fact that sheafification of multiplicative presheaves preserves countable products (\cref{product-sheafification}).
\end{proof}

\section{Postnikov towers in replete topoi}

In this section, our goal is to prove \cref{mainthm}. We begin by summarizing some of the necessary background on the homotopy theory of $\infty$-topoi (following \cite[\S~6.5]{HTT}) that we will need and fix some notation along the way.

\subsection{Preliminaries on \texorpdfstring{$\infty$}{oo}-topoi}
Let $\cT$ be a small $\infty$-category. We denote $\PShv_\infty(\cT) \colonequals \Fun(\cT^\op,\Ani)$. In this paper, by an $\infty$-topos, we will mean an $\infty$-category $\cX$ for which there exists a small $\infty$-category $\mathcal{T}$ and an accessible left exact localization functor $\PShv_\infty(\cT) \to \cX$ (see \cite[Def.~6.1.0.4]{HTT}).
We refer to \cite{HTT} for a detailed presentation of the theory of $\infty$-topoi; in particular, see \cite[Sec.~6.1.1]{HTT} for an $\infty$-categorical analogue of Giraud’s axioms, which gives an intrinsic definition of an $\infty$-topos. 
Here, we will only discuss some examples and constructions that will be useful to us in \cref{BS}.
\begin{example}\label{example-anima}
The $\infty$-category $\Ani$ of spaces (also known as ``anima'') is an $\infty$-topos.
\end{example}
\begin{example}\label{example-presheaves}
Let $\cT$ be a small $\infty$-category.
Then $\PShv_\infty(\cT)$ is an $\infty$-topos, called the \emph{$\infty$-category of $\Ani$-valued presheaves or presheaves of spaces on $\cT$}.
\end{example}
\begin{example}
Let $\cT$ be a Grothendieck site. In this context, one can define the notion of $\Ani$-valued sheaves or sheaves of spaces on $\cT$ denoted as $\Shv_\infty(\cT)$. There is a natural inclusion functor $\Shv_\infty(\cT)  \hookrightarrow \PShv_\infty(\cT)$.
The $\infty$-category $\Shv_\infty(\cT)$ is an $\infty$-topos. The inclusion $\Shv_\infty(\cT) \hookrightarrow \PShv_\infty(\cT)$ admits a left adjoint, given by an accessible left exact localization functor $\PShv_\infty(\cT) \to \Shv_\infty(\cT)$, called \emph{sheafification}.
\end{example}
If $\cX$ is an $\infty$-topos, then so is the slice category $\cX_{/X}$ for any $X \in \cX$ \cite[Lem.~6.3.5.1]{HTT}.
An $\infty$-topos $\cX$ is a presentable $\infty$-category \cite[Prop.~5.5.4.15]{HTT} and thus admits all small limits and colimits \cite[Cor.~5.5.2.4, Def.~5.5.0.1]{HTT}.
In particular, it admits a final object $* \in \cX$. Moreover, for any presentable $\infty$-category $\cC$, there is a notion of $n$-truncated objects and $n$-truncation functors $\tau_{\le n}$ \cite[Prop.~5.5.6.18]{HTT}.
If $\cX$ is an $\infty$-topos, then the subcategory of discrete objects $\tau_{\le 0}(\cX)$ is a $1$-topos \cite[Rem.~6.4.1.3, Thm.~6.4.1.5]{HTT}.
For example, if $\cX = \Shv_\infty(\cT)$ for some Grothendieck site $\cT$, then $\tau_{\le 0}(\cX) \simeq \Shv(\cT)$.

The notion of truncated objects allows one to formulate the notion of Postnikov towers in any presentable $\infty$-category $\cC$.
The following definitions are taken from \cite[Def.~5.5.6.23]{HTT};
we include them here for the convenience of the reader.
Let $\ZZ_{\ge 0}^{\infty}$ be the linearly ordered set $\ZZ_{\ge 0} \cup \left \{ \infty \right \}$ with $\infty$ as the largest element.
The nerve $N(\ZZ_{\ge 0}^{\infty})$ of $\ZZ_{\ge 0}^{\infty}$ is an $\infty$-category.
\begin{definition}\label{conv1}
A \emph{Postnikov tower} in $\cC$ is a functor $X \colon N(\ZZ_{\ge 0}^{\infty})^{\op} \to \cC$ such that for each $n \ge 0$, the map $X(\infty) \to X(n)$ exhibits $X(n)$ as an $n$-truncation of $X(\infty)$.
A \emph{Postnikov pretower} is a functor $X \colon N(\ZZ_{\ge 0})^{\op} \to \cC$ such that for each $n \ge 0$, the map $X(n+1) \to X(n)$ exhibits $X(n)$ as an $n$-truncation of $X(n+1)$.
\end{definition}
Let $\Post^+ (\cC)$ denote the full subcategory of $\Fun(N(\ZZ_{\ge 0}^{\infty})^{\op}, \cC)$ spanned by the Postnikov towers. Let $\Post(\cC)$ denote the full subcategory of $\Fun(N(\ZZ_{\ge 0})^{\op}, \cC)$ spanned by the Postnikov pretowers.
\begin{definition}\label{defconv}
We say that $\cC$ is \emph{Postnikov complete} if the forgetful functor $\Post^+ (\cC) \to \Post(\cC)$ is an equivalence of $\infty$-categories.
\end{definition}
\begin{remark}
In \cite[Prop.~5.5.6.23]{HTT}, Lurie uses the phrase ``Postnikov towers in $\cC$ are convergent'' instead of ``$\cC$ is Postnikov complete.''
We follow the more recent convention of \cite[Def.~A.7.2.1]{SAG} (\textit{cf.}\ \cite[Rem.~5.5.6.23]{HTT}).
\end{remark}
The following lemma gives a useful restatement of the above definition, which we will use later.
\begin{lemma}[{\cite[Prop.~5.5.6.26]{HTT}}]\label{equivalentpost}
Let $\cC$ be a presentable $\infty$-category.
Then $\cC$ is Postnikov complete if and only if, for every functor $X \colon N(\ZZ^\infty_{\ge 0})^{\op} \to \cC$, the following conditions are equivalent:
\begin{enumerate}[label={\upshape{(\alph*)}}]
    \item\label{equivalentpost-tower} The diagram $X$ is a Postnikov tower.
    \item\label{equivalentpost-pretower} The diagram $X$ is a limit in $\cC$ and the restriction $\restr{X}{N(\ZZ_{\ge 0})^{\op}} \colon N(\ZZ_{\ge 0})^{\op} \to \cC$ is a Postnikov pretower.
\end{enumerate}
\end{lemma}
It is a classical result that the $\infty$-topos $\Ani$ is Postnikov complete.
However, this need not be true for a general $\infty$-topos.
A related property, which holds for the $\infty$-topos $\Ani$, but not for a general $\infty$-topos, is the notion of hypercompleteness.
We refer to \cite[\S~6.5]{HTT} for the notion of hypercompleteness, which will be briefly recalled below.
\begin{construction}[{\cite[\S~6.5.1]{HTT}}]\label{relativehomotopy}
Let $\cX$ be an $\infty$-topos and $X \in \cX$ be an object.
Let $S^n$ denote the $n$-sphere and let $* \in S^n$ be a fixed base point.
Evaluation at the base point and \cite[Rem.~5.5.2.6]{HTT} induces a map $s \colon X^{S^n} \to X$, which we identify with an object of $\cX_{/X}$. One can define
\[ \pi_n(X \to *) \colonequals \tau_{\le 0} s \in \cX_{/X}. \]

A similar definition applies to the relative context:
for a map $f \colon X \to Y$ in $\cX$ (viewed as an object of $\cX_{/Y}$), one can use the construction from the previous paragraph to define
\[ \pi_n(f) \in (\cX_{/Y})_{/f} \simeq \cX_{/X}. \]
\end{construction}
\begin{remark}
The projection $S^n \to *$ makes $\pi_n(X \to *)$ from \cref{relativehomotopy} naturally into a pointed object of $\tau_{\le 0} (\cX_{/X})$, which is a group object for $n>0$ and a commutative group object for $n>1$.
See the discussion after \cite[Lem.~ 6.5.1.2]{HTT}.
\end{remark}
\begin{remark}\label{why1}
In the context of \cref{relativehomotopy}, for a base point $x \colon * \to X$, the pullback of $\pi_n(X \to *)$ along $x$ yields a (discrete) pointed object of $\cX$, which we denote as $\pi_n(X, x)$ (or $\pi_n(X, *)$, when the data of the point is clear from the context). 
In the special case $\cX = \Ani$, this recovers the usual $n$-th homotopy group of the pointed space $(X,x)$ \cite[Rem.~6.5.1.6]{HTT}.
For $n = 0$, we define $\pi_0(X) \colonequals \tau_{\le 0}X \in \cX$.
Then $\pi_0(X \to *) \simeq (X \times \pi_0(X) \xrightarrow{\pr_1} X)$ in the category $\cX_{/X}$;
\textit{cf.}\ \cite[Rem.~6.5.1.3]{HTT}.
\end{remark}
\begin{remark}\label{why2}
Let $\cT$ be a Grothendieck site and $\cX \colonequals \Shv_{\infty} (\cT)$. For an object $X \in \cX$ equipped with a base point $x \colon * \to X$, one can also describe the (group) object $\pi_n(X, x)$ for $n > 0$ in the following alternative way:
we can think of $X$ as a pointed sheaf of spaces on $\cT$.
We can then naturally obtain a presheaf of sets denoted as $\pi_n^{\pre}(X, x)$ which sends an object $c \in \cT$ to $\pi_n (X(c), *)$.
For $n=0$, one can similarly define a presheaf $\pi^\pre_0(X)$ via $\cT \ni c \mapsto \pi_0(X(c))$.
By construction and \cite[5.5.6.28]{HTT}, it follows that $\pi_n (X, x)$ (resp.\ $\pi_0(X)$) as defined in \cref{why1} is simply the sheafification of $\pi_n^{\pre}(X, x)$ (resp.\ $\pi^\pre_0(X)$). For a version for unpointed objects, see also \cite[Rem.~6.5.1.4]{HTT} (applied to the geometric morphism $\Shv_\infty(\cT) \hookrightarrow  \PShv_\infty(\cT)$ and the canonical morphism $X \to *$).
\end{remark}
\begin{definition}[{\cite[Def.~6.5.1.10]{HTT}}]\label{infty-conn}
A morphism $f \colon X \to Y$ in an $\infty$-topos $\cX$ is $\infty$-connective if it is an effective epimorphism and $\pi_k(f) = *$ for $k \ge 0$.
For $0 \le n < \infty$, a morphism $f \colon X \to Y$ in $\cX$ is said to be $n$-connective if it is an effective epimorphism and $\pi_k(f)= *$ for $0 \le k <n$.
\end{definition}
\begin{lemma}\label{useful1} Let $\cX$ be an $\infty$-topos. Let $f \colon X \to Y$ be a map in $\cX$ such that $f$ is $(n+1)$-connective and $Y$ is $n$-truncated. Then the map $f$ is an $n$-truncation.  
\end{lemma}
\begin{proof}
Let us consider the map $p \colon X \to \tau_{\le n}X$. By the sequence of pointed objects in \cite[Rmk.~6.5.1.5]{HTT} applied to the maps $X \xrightarrow{p} \tau_{\le n} X \to *$ along with \cite[Rmk.~6.5.1.7]{HTT} and \cite[Lem.~6.5.1.9]{HTT}, it follows that $\pi_k(p) = *$ for $k \le n$ and $\pi_k(p) = \pi_k(X \to *)$ for $k>n$. This implies that the map $p$ is $(n+1)$-connective. Since $Y$ is $n$-truncated, the map $f$ factors through $p$, i.e., there is a map $q \colon \tau_{\le n} X \to Y$ such that there is a diagram
\[ \begin{tikzcd}
                                 & \tau_{\le n}X \arrow[rd, "q"] &   \\
X \arrow[ru, "p"] \arrow[rr,"f"] &                               & Y
\end{tikzcd} \]
in $\cX$.
By \cite[Prop.~6.5.1.16]{HTT}, it follows that $q$ is $(n+1)$-connective.
But since $\tau_{\le n}X$ and $Y$ are both $n$-truncated, $q$ is also $n$-truncated.
Therefore, $q$ is an equivalence \cite[Prop.~6.5.1.12]{HTT}. This finishes the proof.
\end{proof}

\begin{definition}
An $\infty$-topos $\cX$ is called \emph{hypercomplete} if every $\infty$-connective morphism in $\cX$ is an equivalence.
\end{definition}

\begin{remark}\label{hypercomplete1}
Every $\infty$-topos $\cX$ admits a hypercompletion denoted by $\cX^{\wedge}$, which may be characterized by a universal property \cite[Prop.~6.5.2.13]{HTT}.
\end{remark}

\begin{remark}\label{darm1}
An object of an $\infty$-topos $\cX$ is said to be hypercomplete if it is local with respect to the $\infty$-connective morphisms. It follows that $\cX$ is hypercomplete if and only if every object of $\cX$ is hypercomplete.
Any $n$-truncated object of $\cX$ is hypercomplete \cite[Lem.~6.5.2.9]{HTT}.
As a consequence, it follows that if $\cX$ is Postnikov complete, then $\cX$ is hypercomplete.
\end{remark}
It is however much more subtle to give a criterion for Postnikov completeness.
In general, the hypercompleteness of $\cX$ does not imply Postnikov completeness of $\cX$.
To illustrate this, we mention the following example due to Morel--Voevodsky (which we learned from Marc Hoyois);
see also \cite[Warn.~3.11.13]{BGH}.
\begin{example}[{\cite[Ex.~1.30]{MV}}]\label{exam}
Let $G= \prod^\infty_{i=1} \ZZ/2$.
Let $\cT$ be the site of finite $G$-sets and let us consider the $\infty$-topos $\cX \colonequals \Shv_{\infty}(\cT)^{\wedge}$ given by the hypercompletion of $\Shv_{\infty}(\cT)$ (see the discussion before \cite[Lem.~6.5.2.12]{HTT}).
Let $F \in \cX$ be the hypersheafification of the presheaf $\prod^\infty_{i=1} K(\ZZ/2, i)$.
Since hypersheafification preserves truncations \cite[Prop.~5.5.6.28]{HTT}, $\pi_0(F) = \tau_{\le 0}(F)$ is trivial. 
On the other hand, hypersheafification preserves finite limits, so by \cite[Prop.~5.5.6.28, Lem.~6.5.1.2]{HTT} $\lim_{n} \tau_{\le n} F \simeq \prod^\infty_{i=1} K(\ZZ/2, i)^\sharp$, where $K(\ZZ/2, i)^\sharp$ is the (hyper)sheafification of $K(\ZZ/2, i)$.
Thus, $\pi_0 (\lim_{n}\tau_{\le n} F)$ is the sheaf associated with the presheaf
\[ U \mapsto \prod^\infty_{i=1} H^i (U, \ZZ/2) \]
for $U \in \cT$.
When $U = *$ is the final object of $T, $ the element in the latter group obtained taking the products of the pullbacks of $\tau^i \in H^i (\ZZ/2,\ZZ/2)$ along the $i$-th projection $G \to \ZZ/2$ for a generator $\tau \in H^1(\ZZ/2,\ZZ/2)$ gives an element that is not killed by any finite cover.
Therefore, $\pi_0 (\lim_{n} \tau_{\le n}F)$ is nontrivial, showing that Postnikov towers do not converge in the hypercomplete $\infty$-topos $\cX$.
We point out that the cohomology ring $H^* (\ZZ/2, \ZZ/2)$ is a symmetric algebra on $H^1(\ZZ/2, \ZZ/2)$, as already noted in \cref{replete-not-finite}.
\end{example}
Nonetheless, the characterization of Postnikov completeness from \cref{equivalentpost} simplifies for a hypercomplete $\infty$-topos, which will be used later. We thank the referee for pointing this out to us.
\begin{lemma}\label{referee}
    Let $\cX$ be a hypercomplete $\infty$-topos.
    Suppose that whenever $X \colon N(\ZZ^\infty_{\ge 0})^\op \to \cX$ is a limit diagram whose restriction $\restr{X}{N(\ZZ_{\ge 0})^\op} \colon N(\ZZ_{\ge 0})^\op \to \cX$ is a Postnikov pretower, then $X$ is a Postnikov tower in the sense of \cref{conv1}.
    Then $\cX$ is Postnikov complete.
\end{lemma}
\begin{proof}
    By \cref{equivalentpost}, we only need to check that any Postnikov tower $X \colon N(\ZZ^\infty_{\ge 0})^\op \to \cX$ is a limit diagram in $\cX$.
    Let $X' \colon N(\ZZ^\infty_{\ge 0})^\op \to \cX$ be a limit diagram for $\restr{X}{N(\ZZ_{\ge 0})^\op}$ and $X \to X'$ be the natural map induced from the universal property of limits.
    By our hypothesis in the statement, $X'$ is a Postnikov tower.
    In particular, the map $X(\infty) \to X'(\infty)$ induces an isomorphism on $n$-truncations for all $n$.
    This implies that $X(\infty) \to X'(\infty)$ is $\infty$-connective. Since $\mathcal{X}$ is hypercomplete, the latter map must be an equivalence, which implies that $X$ is a limit diagram. 
    This finishes the proof.
\end{proof}
\begin{remark}\label{finiteness}
Let us mention some known criteria for an $\infty$-topos $\cX$ to be Postnikov complete.
If $\cX$ is locally of homotopy dimension $\le n$, then $\cX$ is Postnikov complete \cite[Prop.~7.2.1.10]{HTT}.
Another criterion for Postnikov completeness appears in \cite[Lem.~3.4]{Jar} (see also \cite[Prop.~1.2.2]{Toe}), which relies on certain cohomology vanishing assumptions above a fixed degree.
We point out that the notion of finite cohomological dimension \cite[Def.~7.2.2.18]{HTT} is weaker that finite homotopy dimension, i.e., an $\infty$-topos with homotopy dimension $\le n$ also has cohomological dimension $\le n$ \cite[Cor.~7.2.2.30]{HTT}. 
\end{remark}

\subsection{The question of Bhatt--Scholze}\label{BS}

As recorded in \cref{finiteness}, all the criteria for Postnikov completeness discussed before rely on making certain finiteness assumptions.
However, in many cases, these finiteness assumptions do not hold.
In this subsection, we will prove \cref{mainthm} answering \cref{quesintro} of Bhatt and Scholze, which does not make any such finiteness assumption.

\begin{remark}\label{cry}
As discussed in \cref{darm1}, the hypercompleteness assumption in \cref{quesintro} is necessary and cannot be removed; in fact, as we will show in \cref{counterexample}, the $\infty$-topos of sheaves of spaces on a replete topos need not be hypercomplete.
Moreover, \cref{exam} demonstrates that hypercomplete $\infty$-topoi are in general not Postnikov complete.
\end{remark}
In order to prove \cref{mainthm}, we will use the notion of multiplicative presheaves (\cref{subca}) and also some additional preparations. An important observation is the following proposition which shows that certain naturally occuring presheaves are multiplicative.

\begin{proposition}\label{cocacola}
Let $\cT$ be a Grothendieck site and $F$ be a pointed object of $\Shv_{\infty}(\cT)$.
Then $\pi^{\pre}_n (F, *)$ (as defined in \cref{why2}) is a multiplicative presheaf.
\end{proposition}
\begin{proof}
Note that $F$ represents a limit preserving functor again denoted as 
\[ {F} \colon \Shv_{\infty}(\cT)^\op \to  \Ani. \]
The natural inclusion functor $\Shv(\cT) \to \Shv_{\infty}(\cT)$ preserves coproducts.
Therefore, by restriction, we obtain a product preserving functor
\[ F' \colon \Shv(\cT)^\op \to  \Ani. \]
Since $F$ is a pointed object, so is $F'$.
Taking the associated homotopy groups, we obtain a product preserving functor
\[ \pi_n^{\pre} (F',*) \colon \Shv(\cT)^\op \to \Set. \]
By construction, it follows that $\pi_n^{\pre} (F, *)$ is naturally isomorphic to $\pi_n^{\pre} (F') \circ h^\sharp$, where $h^\sharp \colon \cT \to \Shv(\cT)$ is the map from \cref{leftkan}.
By \cref{local-objects}, we see that $\pi_n^{\pre} (F, *)$ is multiplicative, as desired.
\end{proof}
\begin{proposition}[Milnor sequences in a replete topos]\label{milreplete}
Let $\cX$ be a $1$-topos.
Assume that $\cX$ is replete.
Let us consider the $\infty$-topos $\Shv_{\infty}(\cX)$.
Let $(F_{n})_{n \in \NN}$ be an inverse system of pointed objects of $\Shv_{\infty} (\cX)$. Then we have the following exact sequence (of group objects for $q \ge 1$, of pointed objects for $q=0$)\footnote{
A sequence of pointed objects $\dotsb \xrightarrow{\alpha_{i+1}} (G_i,*) \xrightarrow{\alpha_i} (G_{i-1},*) \xrightarrow{\alpha_{i-1}} \dotsb$ is exact if $\im(\alpha_{i+1}) = \alpha^{-1}_i(*)$ for all $i$.}
in $\cX$:
\begin{equation}\label{milnor}
    * \to {\lim}^{1} \pi_{q+1}(F_n, *) \to \pi_q(\lim F_n, *) \to \lim \pi_q(F_n, *) \to *.
\end{equation}
\end{proposition}

\begin{proof}
Let us explain how to construct such a sequence. We may pick a Grothendieck site $\cT$ such that $\Shv(\cT) \simeq \cX$. Then the objects $F_n$ maybe viewed as pointed sheaves of spaces on $\cT$. By the usual Milnor sequence for the $\infty$-category $\Ani$ \cite[Thm.~IX.3.1]{yellow}, we have the following exact sequece (of presheaves of groups for $q \ge 1$, of pointed presheaves of sets for $q=0$):
\begin{equation}\label{d}
     * \to {\lim}^{1} \pi_{q+1}^{\pre}(F_n, *) \to \pi_q^{\pre}(\lim F_n, *) \to \lim \pi_q^{\pre}(F_n, *) \to *.
\end{equation}
Here, for any pointed (pre)sheaf of spaces $F$ on $\cT$, the set-valued presheaf $\pi_q^{\pre}(F,*)$ denotes the functor $\cT \ni U \mapsto \pi_q(F(U),*)$.
If $F$ is a sheaf, then by \cref{why2}, $\pi_q(F, *)$ is simply the sheafification of $\pi_q^{\pre}(F, *)$.
Moreover, by \cref{cocacola}, we know that since $F$ is a sheaf, the presheaf $\pi_q^{\pre}(F, *)$ is a multiplicative presheaf (\cref{subca}). Therefore, we obtain the desired Milnor sequence \cref{milnor} by sheafifying \cref{d} and using \cref{lim-sheafification} and \cref{lim-sheafification2}.
\end{proof}
\begin{remark}\label{toeremark}
In \cite[p.~21]{Toe}, it is mentioned that in general there is no Milnor sequence for sheaves on a site, because sheafification does not commute with inverse limits in general. 
However, since \cite{Toe} uses the fpqc site, which is replete, sheafification commutes with ``enough'' inverse limits (\cref{lim-sheafification}) and in fact one always has a Milnor sequence in this context.
\end{remark}
Finally, we will need the following lemma.
\begin{lemma}\label{infcon}
Let $\cT$ be a Grothendieck site. Let $f \colon F \to F'$ be a map of sheaves of spaces on $\cT$. Let $0 \le n \le \infty$. Suppose that the following two conditions hold.
\begin{enumerate}[label=\upshape{(\arabic*)}]
    \item\label{infcon-surj} The map $\pi_0 (F) \to \pi_0(F')$ is a surjection of sheaves on $\cT$.
    \item\label{infcon-pi} For every object $U \in \cT$ and every morphism $* \to \restr{F}{\cT_{/U}}$ in $\cT_{/U}$, the induced map on homotopy groups
    \[ \pi_m(f,*) \colon \pi_m \bigl(\restr{F}{\cT_{/U}}, *\bigr) \to \pi_m \bigl(\restr{F'}{\cT_{/U}}, *\bigr) \]
    is a surjection for $m=n$, an isomorphism for $0<m<n$ and a map with trivial kernel (as pointed objects) for $m=0$.
\end{enumerate}
Then the map $f$ is $n$-connective.
\end{lemma}

\begin{proof}
This is essentially obtained by unwrapping \cref{infty-conn} when the topos arises from a Grothendieck site.
To this end, we first note that $f$ is an effective epimorphism.
This is equivalent to showing that the induced map $\pi_0(F) \to \pi_0(F')$ is an effective epimorphism \cite[Prop.~7.2.1.14]{HTT}, which follows from hypothesis \ref{infcon-surj}. Second, we need to prove that $\pi_k(f)= *$ as an object of $\tau_{\le 0}(\Shv_{\infty}(\cT)_{/F})$ for all $0\le k <n$.
We may think of $\pi_k(f)$ as a map $\varphi_k \colon \pi_k(f)^{\tot} \to F$ in $\Shv_{\infty}(\cT)$.
Our goal is to show that $\varphi_k$ is an equivalence.
Let $U \in \cT$.
We obtain an induced map $\varphi_k(U) \colon \pi_k(f)^{\tot}(U) \to F(U)$.
Let $p' \colon * \to F(U)$ be a map.
It corresponds to a map $p \colon * \to \restr{F}{\cT_{/U}}$ in $\Shv_{\infty}(\cT)$.
Let $\Fib_{p,U}$ be the fiber of $\restr{f}{\cT_{/U}} \colon \restr{F}{\cT_{/U}} \to \restr{F'}{\cT_{/U}} $ along $\restr{f}{\cT_{/U}}\circ p \colon * \to \restr{F'}{\cT_{/U}}$.
Then by construction, the fiber of $\varphi_k(U) \colon \pi_k(f)^{\tot}(U) \to F(U)$ along $p' \colon * \to F(U)$ is equivalent to $\pi_k (\Fib_{p,U},p)(U)$. Now by our hypothesis, it follows that $\pi_k(\Fib_{p,U},p)$ is trivial.
This implies that $\pi_k (\Fib_{p,U},p)(U)$ is also trivial. Since $p'$ was arbitrary, this implies that the map $\varphi_k(U)$ is an equivalence. Since $U \in \cT$ was arbitrary, we obtain that $\varphi_k$ is an equivalence, as desired.
\end{proof}
We are now ready to prove our main result.
\begin{proof}[{Proof of \cref{mainthm}}]
We pick a Grothendieck site $\cT$ with $\cX \simeq \Shv(\cT)$.
To show that $\cC \colonequals \Shv_\infty(\cT)^\wedge$ is Postnikov complete, by \cref{referee}, we need to prove that the condition \ref{equivalentpost-pretower} $\implies$ \ref{equivalentpost-tower} from \cref{equivalentpost} is satisfied. 
To this end, let $X \colon N(\ZZ^{\infty}_{\ge 0})^{\op} \to \cC$ be a limit diagram (in other words, $X(\infty) \simeq \lim X(i)$) such that the restriction $\restr{X}{N(\ZZ_{\ge 0})^{\op}}$ is a Postnikov pretower.
We must prove that the natural map $\mu_n \colon X(\infty) \to X(n)$ is an $n$-truncation.
By \cref{useful1} and the fact that $X(n)$ is $n$-truncated, it suffices to show that $\mu_n$ is $(n+1)$-connective.

We first argue that $\mu_n$ is an effective epimorphism.
To see this, note that for any $\mathbf{Z}_{\ge 0}$-indexed inverse system $(Y_i)$ in the $\infty$-category $\mathcal{S}$, the natural map $\pi_0 (\lim Y_i ) \to \lim \pi_0 (Y_i)$ is surjective.
Therefore, since limits of sheaves can be computed at the level of presheaves and limits of presheaves are computed objectwise, the natural map $\pi_0^{\pre}(X(\infty)) \simeq \pi_0^{\pre} (\lim X(i)) \to \lim \pi_0^{\pre}(X(i))$ is surjective as a map of presheaves on $\cT$.
But $\pi_0^{\pre}(X(i))$ is a multiplicative presheaf, so from the repleteness of $\cX$ and \cref{lim-sheafification}, it follows that the map
\[ \pi_0(X(\infty)) \simeq \pi_0(\lim_{}X(i)) \to \lim_{}\pi_0(X(i)) \simeq \pi_0 (X(n)) \]
is surjective.
Thus, $\mu_n$ is an effective epimorphism.

In order to conclude that $\mu_n$ is $(n+1)$-connective, we now apply \cref{infcon}.
Hypothesis \ref{infcon-surj} of \cref{infcon} follows from the previous paragraph.
To check hypothesis \ref{infcon-pi}, one can compute the relevant homotopy groups by using the Milnor sequences from \cref{milreplete}. Note that here we also use the fact that for any $U \in \mathcal T$, the 1-topos $\Shv(\mathcal{T}_{/U})$ is also replete and that the $\lim^1$ of any constant pro-system is trivial. This finishes the proof. 
\end{proof}
We include two applications of \cref{mainthm}; the necessary notations and background are explained before the statements of the corollaries.
\begin{notation}\label{not1}
For any $\infty$-topos $\cX$ and any $\infty$-category $\cC$, we denote by $\Shv_\infty(\cX,\cC)$ the $\infty$-category of $\cC$-valued sheaves \cite[Def.~1.3.1.4]{SAG}.
If $R$ is a connective $\EE_1$-ring and $\cC = \LMod_R$ is the stable $\infty$-category of left $R$-modules (see \cite[\S~7.1.1]{HA}), then $\Shv_\infty(\cX,\LMod_R)$ is naturally equipped with a $t$-structure (by extending the construction of \cite[Prop.~2.1.1.1]{SAG}).
\end{notation}
We thank Peter Haine for pointing out the following corollary of \cref{mainthm}, which generalizes \cite[Prop.~3.3.3]{proet}.
\begin{corollary}
Let $\cX$ be a replete topos and $R$ be a connective $\EE_1$-ring.
Then the stable $\infty$-category $D(\cX,R) \colonequals \Shv_\infty\bigl(\Shv_\infty(\cX)^\wedge,\LMod_R\bigr)$ equipped with the $t$-structure from \cref{not1} is left complete (see \cite[\S~1.2.1]{HA}).
\end{corollary}
\begin{proof}
Let $\Sp$ be the $\infty$-category of spectra.
Set $\Shv_{\infty}(\cX, \Sp)^{\wedge} \colonequals \Shv_{\infty}(\Shv_{\infty}(\cX)^{\wedge},\Sp)$. Using the notations from \cite[\S~1.3.2]{SAG}, we have a conservative, limit preserving functor
\[ \Shv_{\infty}(\cX,\Sp)^{\wedge}_{\ge 0} = \Shv_\infty(\Shv_\infty(\cX)^\wedge,\Sp)_{\ge 0} \xrightarrow{\Omega^\infty} \Shv_\infty(\Shv_\infty(\cX)^\wedge,\Ani) \simeq \Shv_\infty(\cX)^\wedge \]
that commutes with the natural truncation functors.
Since $\cX$ is replete, by \cref{mainthm}, $\Shv_{\infty}(\cX)^{\wedge}$ is Postnikov complete. Now, using the properties of the functor $\Omega^\infty$ stated above, the condition \ref{equivalentpost-tower} $\iff$ \ref{equivalentpost-pretower} from \cref{equivalentpost} continues to hold in $\Shv_{\infty}(\cX, \Sp)^{\wedge}_{\ge 0}$.
Thus $\Shv_{\infty}(\cX,\Sp)^{\wedge}$ is left complete.
For an arbitrary connective $\EE_1$-ring $R$, it similarly follows that $D(\cX,R)$ is left complete because the natural functor $D(\cX,R) \to \Shv_{\infty}(\cX,\Sp)^{\wedge}$ is conservative, limit preserving and $t$-exact (\textit{cf.}\ \cite[Prop.~2.1.1.1.(a)]{SAG} and \cite[Prop.~2.1.0.3.(3)]{SAG}).
\end{proof}
Next, we discuss an application of \cref{mainthm} in the theory of affine stacks \cite[\S~2.2]{Toe}. 
\begin{remark}\label{lastrmk}
In \cite[p.~55]{Toe}, To\"en shows that if $F$ is a pointed connected affine stack over a field, then the natural map $F \to \lim_{n} \tau_{\le n} F$ is an equivalence.
The proof uses the criterion from \cite[Prop.~1.2.2]{Toe} and requires the representability of $\pi_n (F, *)$ by unipotent affine group schemes and certain vanishing of cohomology with coefficients in such group schemes. Since affine stacks are hypercomplete and the fpqc topos (after fixing set-theoretic issues, as in \cite{Toe}) is replete, one obtains as a corollary of \cref{mainthm} the following generalization, which removes the assumptions that the stack $F$ is connected (i.e., $\pi_0 (F,*) \simeq *$) and defined over a field.
\end{remark}

\begin{corollary}\label{lastcor}
Let $F$ be an affine stack over $\Spec B$ for any ring $B$. Then the natural map $F \to \lim_{n} \tau_{\le n} F$ is an equivalence. 
\end{corollary}

\begin{example}\label{counterexample}
We end our paper with an example that shows that there are Grothendieck sites $\mathcal{T}$ for which the $1$-topos $\Shv(\mathcal{T})$ is replete but $\Shv_{\infty}(\mathcal{T})$ is not hypercomplete (and thus not Postnikov complete). To do so, we use a 1-topos due to Dugger--Hollander--Isaksen \cite[Ex.~A.9]{DHI} (modifying a suggestion of Carlos Simpson).
Let $\cT$ be the following Grothendieck site: 
\begin{enumerate}
    \item Objects of $\cT$ are the open subintervals $X_n \colonequals (\tfrac{3^n-1}{2 \cdot 3^n},\tfrac{3^n+1}{2 \cdot 3^n})$, $U_n \colonequals (\tfrac{3^n-1}{2 \cdot 3^n},\tfrac{3^{n+1}+1}{2 \cdot 3^{n+1}})$, and $V_n \colonequals (\tfrac{3^{n+1}-1}{2 \cdot 3^{n+1}},\tfrac{3^n+1}{2 \cdot 3^n})$ of $(0,1)$.
    \item Morphisms are inclusions of underlying sets.
    \item Covers are given by jointly surjective maps.
\end{enumerate}
One may think of the category $\cT$ in terms of the following poset:
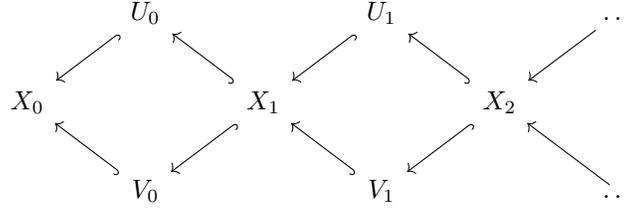
\[ \begin{tikzcd}
& U_0 \arrow[ld,hook] && U_1 \arrow[ld,hook] && \dotso \arrow[ld] \\
X_0 && X_1 \arrow[lu,hook] \arrow[ld,hook] && X_2 \arrow[lu,hook] \arrow[ld,hook] \\
& V_0 \arrow[lu,hook] && V_1 \arrow[lu,hook] && \dotso \arrow[lu]
\end{tikzcd} \]

We claim that the $1$-topos $\Shv(\cT)$ is locally weakly contractible in the sense of \cite[Def.~3.2.1]{proet}.
Indeed, the site $\cT$ is subcanonical and every object is qcqs.
Since every covering family of the objects $U_n$ (resp.\ $V_n$) contains $U_n$ (resp.\ $V_n$), the sheaves $h_{U_n}$ (resp. $h_{V_n}$) are weakly contractible.
Further, given any $Z \in \cT$, there exists a surjection $\left( h_{U_n} \coprod h_{V_n} \right) \twoheadrightarrow h_Z$ for some $n \ge 0$.
This implies that $\Shv(\cT)$ is locally weakly contractible.
In particular, $\Shv(\cT)$ is replete \cite[Prop.~3.2.3]{proet}.

However, $\cX \colonequals \Shv_{\infty}(\cT)$ is not hypercomplete.
In fact, \cite[Ex.~A.9]{DHI} exhibits an example of a non-hypercomplete sheaf $G$ with values in $D(\ZZ)_{\ge 0}$;
it is given by
\[ G(X_n) \colonequals \ZZ[\Sing(S^{n-1})], \quad G(U_n) \colonequals \ZZ[\Sing(D^n_+)], \quad G(V_n) \colonequals \ZZ[\Sing(D^n_-)], \]
where $D^n_+$ (resp.\ $D^n_-$) denotes the upper (resp.\ lower) hemisphere of $S^n$.
We use the convention that $S^{-1} \colonequals \varnothing$; in particular, $\ZZ[\Sing(S^{-1})] = 0$.
The maps induced by the standard inclusions of topological spaces $D^n_+ \hookrightarrow S^n$ and $D^n_{-} \hookrightarrow S^n$ naturally equip $G$ with the structure of a presheaf.
The sheaf property amounts to checking \v{C}ech descent, which essentially follows from the Mayer--Vietoris theorem for the open cover $S^n = D^n_+ \cup D^n_-$.
To show that $G$ is not hypercomplete, \cite[Ex.~A.9]{DHI} constructs an explicit hypercover for which $G$ does not satisfy descent.
Alternatively, one can simply observe that for all $q > 0$, we have
\[ \pi_q(G) \simeq \bigl(\pi^\pre_q(G)\bigr)^\sharp \simeq 0. \]
Here, the latter isomorphism follows from the isomorphisms $H_i (D^n_{+}, \ZZ) \simeq H_i (D^n_{-},\ZZ)=0$ for $i>0$ and the fact that every object of the site $\cT$ admits a refinement by $U_n$ and $V_n$. Therefore, the natural map $G \to \tau_{\le 0} G$ is $\infty$-connective. However, it cannot be an equivalence since $G$ is not $0$-truncated (e.g., $\pi_n(G(X_{n+1})) \simeq H_n (S^n, \ZZ) \simeq \ZZ$).
This shows that $G$ is not hypercomplete.
\end{example}

\bibliographystyle{amsalpha}
\bibliography{references}

\end{document}